\documentclass[12pt,reqno]{amsart}
\usepackage{latexsym, amsmath, amssymb, longtable, booktabs,amscd,microtype,booktabs,cases}
\usepackage{amsaddr}

\usepackage[english]{babel}
\usepackage[utf8x]{inputenc}
\usepackage{hyperref}
\usepackage[utf8x]{inputenc}
\usepackage{graphicx}
\usepackage{cleveref}
\numberwithin{equation}{section}
\usepackage[switch,columnwise]{lineno}
\usepackage{lipsum} 
\usepackage{tikz-cd}
\usetikzlibrary{decorations.markings,arrows}

\usepackage{tikz}
\usepackage{scalefnt}
\usetikzlibrary{positioning}
\usepackage{stackengine}
\usepackage{fancyhdr,lipsum}
\begin{document}

\newcommand\A{\mathbb{A}}
\newcommand\C{\mathbb{C}}
\newcommand\G{\mathbb{G}}
\newcommand\N{\mathbb{N}}
\newcommand\T{\mathbb{T}}
\newcommand\E{{\mathbb{E}}}

\newcommand{\cA}{\mathcal{A}}
\newcommand{\cB}{\mathcal{B}}
\newcommand{\cC}{\mathcal{C}}
\newcommand{\cS}{\mathcal{S}}
\newcommand{\cT}{\mathcal{T}}

\newcommand\cO{\mathcal{O}}
\newcommand\cE{{\mathcal{E}}}
\newcommand\cF{{\mathcal{F}}}
\newcommand\cG{{\mathcal{G}}}
\newcommand\GL{{\mathrm{GL}}}
\newcommand\HH{{\mathrm H}}
\newcommand\mM{{\mathrm M}}

\newcommand\fS{\mathfrak{S}}
\newcommand\fP{\mathfrak{P}}
\newcommand\cP{\mathcal{P}}
\newcommand\fp{\mathfrak{p}}
\newcommand{\fa}{\mathfrak{a}}
\newcommand\fQ{\mathfrak{Q}}
\newcommand{\fN}{\mathfrak{N}}
\newcommand\cQ{\mathcal{Q}}
\newcommand\Qbar{{\bar{\Q}}}
\newcommand\sQ{{\mathcal{Q}}}
\newcommand\sP{{\mathbb{P}}}
\newcommand{\Q}{\mathbb{Q}}
\newcommand{\tH}{\mathbb{H}}
\newcommand{\Z}{\mathbb{Z}}
\newcommand{\R}{\mathbb{R}}
\newcommand{\F}{\mathbb{F}}
\newcommand\gP{\mathfrak{P}}
\newcommand\Gal{{\mathrm {Gal}}}
\newcommand\SL{{\mathrm {SL}}}
\newcommand\Hom{{\mathrm {Hom}}}
\newcommand{\legendre}[2] {\left(\frac{#1}{#2}\right)}
\newcommand\iso{{\> \simeq \>}}
\newtheorem{thm}{Theorem}[section]
\newtheorem{theorem}[thm]{Theorem}
\newtheorem{cor}[thm]{Corollary}
\newtheorem{conj}[thm]{Conjecture}
\newtheorem{prop}[thm]{Proposition}
\newtheorem{lemma}[thm]{Lemma}
\newtheorem{definition}[thm]{Definition}
\newtheorem{remark}[thm]{Remark}
\newtheorem{example}[thm]{Example}
\newtheorem{claim}[thm]{Claim}
\newtheorem{lem}[thm]{Lemma}
\newtheorem*{fact}{Fact}
\newtheorem{note}{Note}
\newtheorem*{Thm}{Theorem}

\title{Resistance distance in connected balanced digraphs} 
\author{R. Balakrishnan}
\address{Department of Mathematics, Bharathidasan University, Tiruchirappalli-620024, India,\\Email: mathrb13@gmail.com}
\author{S. Krishnamoorthy}
\address{Indian Institute of Science Education and Research Thiruvananthapuram, Maruthamala P.O.,Vithura, Thiruvananthapuram-695551, Kerala, India, Email: srilakshmi@iisertvm.ac.in}
%\author{R. Balakrishnan\textsuperscript{a}}
%\address{\textsuperscript{a} Department of Mathematics, Bharathidasan University, Tiruchirappalli-620024, India, mathrb13@gmail.com}
%\author{S. Krishnamoorthy\textsuperscript{b}}
%\address{\textsuperscript{b}Indian Institute of Science Education and Research Thiruvananthapuram,\\ Maruthamala P.O.,Vithura, Thiruvananthapuram-695551, Kerala, India, srilakshmi@iisertvm.ac.in}
\author{W. So}
%\address{Department of Mathematics, Bharathidasan University, Tiruchirappalli, 620024, India}
%\email{mathrb13@gmail.com}
%\address{Indian Institute of Science Education and Research Thiruvananthapuram, Maruthamala P.O., Vithura, Thiruvananthapuram-695551, Kerala, India.}
%\email{srilakshmi@iisertvm.ac.in}
%\author[$\dagger$]{R. Campbell}
%\author[$\star$]{M. Dane}
%\author[$\ddag$]{J. Jones}
%\affil[$\dagger$]{Department of Mathematics, Pennsylvania State University,Pittsburgh, Pennsylvania 13593}
%\affil[$\star$]{Atmospheric Research Station,
%Pala Lundi, Fiji}
%\affil[$\ddag$]{Department of Philosophy, Freedman College,
%Periwinkle, Colorado 84320}
\address{Department of Mathematics, San Jose State University, San Jose, CA 95192-0103, United States, Email :wasin.so@sjsu.edu}
\keywords{}
\subjclass[2010]{}

\begin{abstract}
Let $D = (V, E)$ be a strongly connected and balanced digraph with vertex set $V$ and arc set $E.$ The classical
distance $d_{ij}^D$ from $i$ to $j$ in $D$ is the length of a shortest directed path from $i$ to $j$ in $D.$  Let $L$ be the Laplacian matrix of $D$ and $ L^{\dagger} = ( l_{ij}^{\dagger} )$ be the Moore-Penrose inverse of $L.$ The resistance
distance  from $i$ to $j$ is then defined by $r_{ij}^D :=   l_{ii}^{\dagger } +  l_{jj}^{\dagger } - 2  l_{ij}^{\dagger }.$ %Let $\{ D_1, D_2, ...., D_k  \}$ be a sequence of strongly connected balanced digraphs with $D_i   \cap D_j$ having at most one %vertex in common for all $i \neq j$ and with $r_{ij}^{D_t} \leq d_{ij}^{D_t}  \ \forall \  t = 1 \ \mathrm{to} \ k.$
  Let $\mathcal{C}$ be a collection of connected, balanced digraphs, each member of which is a finite union of the form $D_1 \cup D_2 \cup ....\cup D_k$ where each $D_i$ is a connected and balanced digraph with $D_{i} \cap ( D_1 \cup D_2 \cup ....\cup D_{i-1} )$ being a single vertex, for all $i,$  $1 < i \leq k.$ In this paper, we show that for any digraph $D$ in $\mathcal{C}$, $r_{ij}^D \leq d_{ij}^D \ (*)$. This is established by partitioning the Laplacian matrix of $D$.  This generalizes the main result in \cite{BBG2}.  As a corollary, we deduce a simpler proof of the result in \cite{BBG2}, namely, that for any directed cactus $D$, the inequality (*) holds. Our results provide an affirmative answer to a well known interesting conjecture ( cf : Conjecture \ref{conj} ).

\end{abstract}

\keywords{Strongly connected balanced digraphs, Laplacian matrix, Moore-Penrose inverse, resistance distance.}
\subjclass[2010]{Primary: 05C50}
\maketitle
\section{introduction}
Let $D = (V, E) $ be a directed graph with vertex set $V = \{1, 2,..., n \}$ and 
arc set $E$.  %\\$E =  \{ (i,j) \mid \exists \ \mathrm{ an \ arc \  from} \ i \ \mathrm{to} \  j \}.$
The indegree and outdegree of a vertex $i$ in $D$ are defined by
$$ \delta^{\mathrm{in}}_{i} = | \{ j \mid (j, i) \in E \}|, \ \delta^{\mathrm{out}}_{i} = | \{ j \mid (i, j) \in E \}|$$ respectively.
The Laplacian matrix of the digraph $D$ is the matrix  $L = (l_{ij})$ of order $n$, where
 
$$l_{ij} = \begin{cases}  \delta^{\mathrm{out}}_{i}  \ \mathrm{if} \ i = j \\
 -1 \ \mathrm{if} \  (i, j) \in E \\
 0 \ \mathrm{otherwise}.
 \end{cases}
 $$

 \begin{definition}
 A digraph $D = (V,E)$ is balanced if $\delta^{\mathrm{in}}_{i} =\delta^{\mathrm{out}}_{i}$ for all $i \in V$.
 \end{definition}
 \begin{definition}
 A digraph $D$ is strongly connected if there is a directed path from any vertex of $D$ to any other vertex of $D$.
 \end{definition}

 Throughout this paper, we consider only finite digraphs.\\
 Let $A$ be a complex $m \times n$ matrix. The Moore-Penrose inverse $A^{\dagger } = ( a_{ij}^{\dagger })$ of $A = (a_{ij})$ is the unique matrix of size $n \times m$ which satisfies  
 $A A^{\dagger } A = A, \ \
 A^{\dagger }AA^{\dagger }=A^{\dagger },\ \
 (AA^{\dagger })^{*} = AA^{\dagger }, \ \
 (A^{\dagger }A)^{*} = A^{\dagger }A,$ where $*$ denotes conjugate transpose.
  The Moore-Penrose inverse exists for any arbitrary matrix [\cite{BG}, Chapter 1], and when a square matrix has a regular inverse, this inverse is its Moore-Penrose inverse. A simple way to compute the Moore-Penrose is by using the singular value decomposition. 
  %Let
%$A = U \Sigma V^*$ 
%be the singular value decomposition of 
%$A$, where $\Sigma$ is a rectangular diagonal matrix. Then 
%$
%A^{+}=V \Sigma ^{+}U^{*},$  $\Sigma ^{+}$ can be computed by taking the reciprocal of %each non-zero element on the diagonal, leaving the zeros in place, and then transposing the matrix. 
More on the theory of generalized inverses of matrices can be found in \cite{BG}, \cite{CSMC}. Klein and Radic introduced the notion of resistance distance in undirected graphs \cite{KR}. Balaji, Bapat and Goel who introduced resistance distance for digraphs in \cite{BBG1}. The resistance distance from the vertex $i$ to the vertex $j$ of a digraph $D$ is denoted as $r_{ij}^D$. Then the resistance distance matrix $R = (r_{ij}^D)$ of the digraph $D$ is the square matrix of order $n,$ where 
 $r_{ij}^D =  l_{ii}^{\dagger } +  l_{jj}^{\dagger } - 2  l_{ij}^{\dagger },$ $l_{jj}^{\dagger }$ is the $(i,j)$ entry of the Moore-Penrose inverse of the Laplacian matrix $L.$ The study of resistance distance arose from the fact that a 
simple, connected graph is a representation of an electrical network with unit resistance placed on each of its edges. 
\\The validity of Conjecture \ref{conj} given below has been verified to be true in several examples. 
\begin{conj}\cite{BBG2}\label{conj}
For a strongly connected balanced digraph $D$, the resistance distance $r_{ij}^{D}$ is upper bounded by $d_{ij}^{
D}$ i.e., $r_{ij}^{D} \leq d_{ij}^{
D},$ where $d_{ij}^{D}$ is the length of a shortest directed path from $i$ to $j$ in $D$. 
\end{conj} 
Any connected balanced digraph is strongly connected. Hence the following conjecture\\( Conjecture \ref{conj2} ) is equivalent to Conjecture \ref{conj}.
\begin{conj} \label{conj2}
For a connected balanced digraph $D$, the resistance distance $r_{ij}^{D}$ is upper bounded by $d_{ij}^{
D}$ i.e., $r_{ij}^{D} \leq d_{ij}^{
D},$ where $d_{ij}^{D}$ is the length of a shortest directed path from $i$ to $j$ in $D$. \end{conj} 

  %Let $\{ D_1, D_2, ...., D_k  \}$ be a %sequence of strongly connected balanced %digraphs with $D_l$, $ 1\leq l \leq k $ %having exactly one vertex in common with $D = % D_1 \cup D_2  .....  \cup D_k $ and with %$D_{i-1} \cap D_i = \{ \mathrm{a} \ %\mathrm{single} \ \mathrm{vertex} \}$ as %shown in the following figure. 
  %\begin{center}
%Figure. Connected, balanced digraphs %(consecutive graphs have a single vertex in %common).
%\end{center}
%\vskip5pt
% \begin{center}
  %\tikzset{  
           % auto, node distance=1.5cm,very thick, 
           % node/.style = {circle, draw, inner sep=0.05cm, fill=black},
          %  point/.style= {circle, draw, inner sep=0.01cm, fill=gray}
%            }    
   % \begin{tikzpicture}
    
    %    \node (d1) at (0,0) [node,label=below:$D_1$] {};
    %    \node (d2) at (2,0) [node,label=below:$D_2$] {};
    %    \node (d3) at (4,0) [node,label=below:$D_3$] {};
     %   \node (p1) at (4.5,0) [point]{} ;
      %  \node (p2) at (5,0) [point] {} ;
      %  \node (p3) at (5.5,0) [point] {};
      %  \node (p4) at (6,0) [point] {};
       % \node (p5) at (6.5,0) [point] {};
    %    \node (d4) at (7,0) [node,label=below:$D_{k-1}$] {};
     %   \node (d5) at (9,0) [node,label=below:$D_k$] {};
      %  \path (d1) edge[bend right] (d2) edge[bend left] (d2);
       % \path (d2) edge[bend right] (d3) edge[bend left] (d3);
        %\path (d4) edge[bend right] (d5) edge[bend left] (d5);
    
    %\end{tikzpicture}  

%\end{center}
\begin{definition}
A directed cactus is a strongly connected balanced digraph in which each edge is contained in exactly one directed cycle.
\end{definition} 
Conjecture \ref{conj} is true for any directed cactus \cite{BBG2}. This was proved by using combinatorial methods by counting the number of spanning
forests satisfying certain conditions.
There is no proof showing that for a general connected, balanced digraph $D$ which is not a directed cactus, which Conjecture \ref{conj} is true.
i.e. $r_{ij}^D \leq d_{ij}^D$ for all $i$ and $j$.
In this paper, we construct a new category $\mathcal{C}$ (which strictly contains the set of all directed cactus) and prove the following theorem by using the method of partitioning of matrices.
%\begin{prop}\label{mainprop}
%Let $D_1, D_2 $ be connected balanced arc-disjoint digraphs and let $D_1 \cap D_2  $ be a single vertex. Suppose Conjecture \ref{conj} is true for both $D_1$ and $D_2$.
%It is enough to prove the theorem  \ref{maintheorem} for the subcategory
%$\mathcal{C^{'}}= \{ D \mid D = D_1 \cup D_2 \}$ where 
%$D_1 \cup D_2$ is the union of $D_1$ and $D_2$ and $V(D_1) 
%\cap V(D_2) = \{ \mathrm{a} \ \mathrm{single}  \ \mathrm{vertex} \}$.
%By induction hypothesis, the theorem \ref{maintheorem} holds for $\mathcal{C}$.
% Let $(i,j)$ be an arc.
%Then $r_{ij}^{D} \leq 1$ if either $D_1  \cap D_2 = \{ i \}$ or  
%$\{ j \}.$\\
\begin{thm}\label{mainthm}
 Let $\mathcal{C}$ be a collection of connected, balanced digraphs, each member of which is a finite union of the form $D_1 \cup D_2 \cup ....\cup D_k$ where each $D_i$ is a connected and balanced digraph with $D_{i} \cap (D_1 \cup D_2 \cup ....\cup D_{i-1} )$ being a single vertex, for all $i,$ $ 1 < i \leq k.$
Then $r_{ij}^D \leq d_{ij}^D$ for all digraphs $D$ in $\mathcal{C}$.
\end{thm}
Note that $r_{ij}^D$  may be different from $r_{ij}^{D_t}$. (See Remark \ref{note} below).
 Theorem \ref{mainthm} provides the proof of Conjecture
 \ref{conj} for all digraphs $D$ in $\mathcal{C}$.
In this paper, as a corollary of Theorem \ref{mainthm}, we deduce the proof of Conjecture \ref{conj} for all directed cacti, which in fact is the main result in \cite{BBG2}. 
%Also, our proof turns out to be shorter than the original proof given in \cite{BBG2}. 

 \begin{cor}[\cite{BBG2}, Theorem 3.5] \label{maintheorem}
 Let $D $ be a directed cactus. Then
 $r_{ij}^D \leq d_{ij}^{D}$ for all $i$ and $j.$
 \end{cor}
\begin{cor} 
Let $D_1$ and $D_2$ be two directed cacti with  $D_1 \cap D_2 $ being a single vertex. f
Then $D_1 \cup D_2$ also satisfies the condition that 
$r_{ij}^D \leq d_{ij}^D$ for all $i$ and $j.$
\end{cor} 
%A consequence of Theorem \ref{maintheorem} is the main result of the paper { \cite{BBG2}, Theorem 3.5 ).\begin{cor}
%Let $D$ be a directed cactus graph. Then $r_{ij}^D \leq d_{ij}^D.$
%\end{cor}
\section{Preliminaries} We state some lemmas which will be used in the proof of our main theorem.\\
 For a square matrix $A$ of order $N,$ let $A^{'}$ be the transpose of $A$. Let $S_1$ be a set of rows, $S_2$ be a set of columns of $A$ and let $A[  {S_1}^{c},  {S_2}^{c}  ]$  be the matrix obtained by deleting the rows of $S_1$ and columns of $S_2$. If $A$ is a square matrix with each row sum and each column sum being zero, then 
all the cofactors of $A$ are the same. In particular, if $\mathrm{rank} ( A ) = N-1,$ then $\mathrm{det} (A[  \{ i \}^{c},  \{ i \}^{c}  ] )$ are the same  and also non-zero for all $1 \leq i \leq N$. Hence $A[  \{ i \}^{c},  \{ i \}^{c}  ]$ is non-singular.\\
A square matrix $A$ is called a $\mathbb{Z}$-matrix, if every off-diagonal entry of $A$ is non- positive.
%The following lemma is from \cite{BBG1}, Lemma 2.7.
\begin{lemma} [ \cite{BBG1}, Lemma 3.1 ]\label{mainlemma}
Let $L$ be a $\mathbb{Z}$-matrix of order $N$ such that $L1 = L^{'}1 = 0$ and $\mathrm{rank}(L) = N-1$. Then $L$ can be partitioned as 
$L = \begin{bmatrix} B & -Be \\
-e^{'}B & e^{'}Be 
\end{bmatrix}$ and 
$$L^{\dagger}  =
%=\begin{bmatrix} A  & B \\ C  & D \end{bmatrix} 
\begin{bmatrix}  B^{-1} - \frac{ e e^{'} B^{-1} } {N}  - \frac{ B^{-1} e e^{'} }{ N }   &   - \frac{ B^{-1} e } {N} \\
 - \frac{ e^{'} B^{-1} } { N }  & 0  
\end{bmatrix} +   \frac{ e^{'} B^{-1} e 1 1^{'} }{ N^2 }   ,$$
where $B$ is a square matrix of order $N-1$ and $e={\bf1}_{N-1}.$
\end{lemma}
%It is clear that $B = L [ {n}^c, {n}^c ]$ in the above lemma.
As an immediate application of the above lemma, the following theorem is obtained in \cite {BBG1}.
\begin{thm}  [ \cite {BBG1}, Theorem 3.1 ] \label{theorem-rij}
Let $D$ be a strongly connected, balanced digraph. 
The resistance distance $r_{ij}^D$ satisfies\\
(i) $r_{ij}^D \geq 0,$ $\forall \ i$ and $j$, $r_{ij}^D = 0 \Leftrightarrow i = j,$
(ii) Triangle inequality : $r_{ij}^D \leq r_{ik}^D + r_{kj}^D$ for any $i, j$ and $k.$
\end{thm}
The resistance distance is not a metric in general as it need not satisfy the symmetry property.\\

Let $D=(V,E)$ be a connected, balanced digraph with vertex set $V$ and arc set $E$. Let $\kappa(D, i)$ be the number of spanning trees of $D$ rooted at a vertex $i$ of $D.$ By all minors matrix Theorem \cite{CH}, it follows that
$\kappa(D, i) = \mathrm{det}(L[\{i\}^c, \{i\}^c]).$
Since $\mathrm{rank}(L) = |V| -1$ and $L1 = {L}^{'}1 = 0$, all the cofactors of $L$ are equal. Hence $\kappa(D, i)$ is independent of $i$ and we denote it by $\kappa(D).$ 
\begin{lemma} [\cite{BBG2}, Lemma 3.1] \label{lemma6}
Let $D = (V, E)$ be a strongly connected, balanced digraph with vertex set $V$ and
arc set $E$. Let $|V| = N$ and $i, j \in V$. If $(i,j) \in E$ or $(j, i) \in E,$ then
$\mathrm{det} \left(L[ \{i,j \}^c, \{i, j\}^c ]\right) \leq \kappa(D),$ where $L$ is the Laplacian of $D$ and $\kappa(D)$ is the number of spanning trees of $D$ rooted at $i.$\end{lemma}
%An explicit formula for the Moore-Penrose inverse of a partitioned matrix is obtained in
%\cite{CM}. Our main tool is the following lemma.
%\begin{lemma}\label{lemma7}
%f $M$ is an $m \times n$ matrix partitioned with
%$M = \begin{bmatrix} 
%A & O \\
%O & C 
%\end{bmatrix},$ then the Moore-Penrose inverse of $M$ is given by
%$M^{\dagger} = \begin{bmatrix} A^{\dagger} & O \\
%O & C^{\dagger}
%\end{bmatrix}$.
%\end{lemma}
%\begin{proof}
%If $M = \begin{bmatrix} 
%A & O \\
%B & C 
%\end{bmatrix},$ then by Lemma 2 of \cite{CM}, $M^{\dagger} = \begin{bmatrix} K^{\dagger}A^{'} & K^{\dagger}B^{'} \\
%O & C^{\dagger}
%\end{bmatrix}$, where $K = A^{'}A + B^{'}B$.\\
%Taking $B$ as zero matrix, 
%we get $M^{\dagger} = \begin{bmatrix} (A^{'}A)^{\dagger}A^{'} & O \\
%O & C^{\dagger}
%\end{bmatrix}$.
%By using the fact that$(A^{'}A)^{\dagger}A^{'} = A^{\dagger}$ [Chapter 1, \cite{BG}],  we obtain 
% $M^{\dagger} = \begin{bmatrix} A^{\dagger} & O \\
%O & C^{\dagger}
%\end{bmatrix} $.
%\end{proof}

\begin{lemma} \label{lemma7}Let $A^{\dagger}$ and $C^{\dagger}$ be the Moore-Penrose inverses of $A$ and $C$ respectively.
If $M$ is an $m \times n$ matrix partitioned with $M=\left[  \begin{array}{cc} A &0 \\0&C \end{array} \right]$  then $M^{\dagger} = 
\left[  \begin{array}{cc} A^{\dagger} &0 \\0&C^{\dagger} \end{array} \right]$.
\end{lemma}
\begin{proof} By definition, $A^{\dagger} A A^{\dagger} = A^{\dagger}, A A^{\dagger} A =A,
(AA^{\dagger})^*=AA^{\dagger},  (A^{\dagger}A)^*=A^{\dagger}A$,
 and $C^{\dagger} C C^{\dagger} = C^{\dagger}, CC^{\dagger} C =C,
(CC^{\dagger})^*=CC^{\dagger},  (C^{\dagger}C)^*=C^{\dagger}C$.  Now

\[
\left[  \begin{array}{cc} A^{\dagger} &0 \\0&C^{\dagger} \end{array} \right]
\left[  \begin{array}{cc} A &0 \\0&C \end{array} \right]
\left[  \begin{array}{cc} A^{\dagger} &0 \\0&C^{\dagger} \end{array} \right]
= 
\left[  \begin{array}{cc} A^{\dagger} A A^{\dagger}&0 \\0&C^{\dagger} CC^{\dagger}\end{array} \right]
=
\left[  \begin{array}{cc} A^{\dagger} &0 \\0&C^{\dagger} \end{array} \right] 
\]

\[ 
\left[  \begin{array}{cc} A &0 \\0&C \end{array} \right]
\left[  \begin{array}{cc} A^{\dagger} &0 \\0&C^{\dagger} \end{array} \right]
\left[  \begin{array}{cc} A &0 \\0&C \end{array} \right]
=
\left[  \begin{array}{cc} AA^{\dagger}A &0 \\0&CC^{\dagger}C \end{array} \right]
=
\left[  \begin{array}{cc} A &0 \\0&C \end{array} \right]
\]

\[
\left( \left[  \begin{array}{cc} A &0 \\0&C \end{array} \right] \left[  \begin{array}{cc} A^{\dagger} &0 \\0&C^{\dagger} \end{array} \right] \right)^*
= \left[  \begin{array}{cc} AA^{\dagger} &0 \\0&C C^{\dagger} \end{array} \right]^*
= \left[  \begin{array}{cc} (AA^{\dagger} )^*&0 \\0& (C C^{\dagger})^* \end{array} \right] \]
\[ = \left[  \begin{array}{cc} AA^{\dagger} &0 \\0&C C^{\dagger} \end{array} \right]
= \left[  \begin{array}{cc} A &0 \\0&C \end{array} \right] \left[  \begin{array}{cc} A^{\dagger} &0 \\0&C^{\dagger} \end{array} \right]
\]

\[ \mathrm{Similarly}
\left( \left[  \begin{array}{cc} A^{\dagger} &0 \\0&C^{\dagger} \end{array} \right]  \left[  \begin{array}{cc} A &0 \\0&C \end{array} \right] \right)^*
= \left[  \begin{array}{cc} A^{\dagger} &0 \\0&C^{\dagger} \end{array} \right]  \left[  \begin{array}{cc} A &0 \\0&C \end{array} \right]  
\]
Hence $M^{\dagger} = 
\left[  \begin{array}{cc} A^{\dagger} &0 \\0&C^{\dagger} \end{array} \right]$ by definition.

\end{proof}
\begin{lemma}\label{sublemma}
Let $D = (V,E) $ be a strongly connected balanced digraph. 
Suppose $\forall$ $(i,j) \in E $, $r_{ij}^{D} \leq 1$, then Conjecture 
\ref{conj} is true for $D$. 
\end{lemma}
\begin{proof}
Suppose $i$ and $j$ are two distinct vertices of the digraph $D$. Let
a shortest path from $i$ to $j$ be given by
$i = i_0 \rightarrow i_1 \rightarrow .......\rightarrow i_{k-1} \rightarrow i_{k} = j$.
The length of the path is $k = d_{ij}^D$.
We have $r_{i_{l-1}i_{l}}^D \leq 1 \ \forall \  l = 1 \ \mathrm{to} \ k$.
By repeatedly using the triangle inequality [Theorem \ref{theorem-rij}, (ii)], we get
$r_{ij}^D \leq r_{i i_1}^D + r_{i_1i_2}^D + .....+r_{i_{k-1}j}^D
\leq 1 + 1 + ....+1$ (upto $k$ times).
Hence $r_{ij}^D \leq d_{ij}^D $.
\end{proof}
\begin{lemma} [\cite{BBG2}, Lemma 3.2, case(i)]\label{BBG lemma}
Let $D=(V,E)$ be a strongly connected, balanced digraph. 
Let $(i,j) \in E$. If both $i$ and $j$ are of indegree one, then
$r_{ij}^D \leq 1.$
\end{lemma}
\begin{lemma} [\cite{BBG1}, Theorem 5.2, (ii)] \label{BBG2 lemma} 
Let $D$ be a strongly connected balanced digraph with vertex set $\{ 1,2,...,n \},$ Laplacian matrix $L$ and
Resistance matrix $R.$ Then for each distinct pair $i,j \in \{1,2,...,n\},$ $$r_{ij}^{D} + r_{ji}^D = 2 \frac{\mathrm{det}L[ \{i,j\}^c, \{i,j\}^c]} {\kappa(D)}.$$
    \end{lemma}
%\begin{proof}
%Refer to the proof of.
%\end{proof}
\begin{remark}\label{note}
Let $D$ be a strongly connected digraph and $D_1,$ a strongly connected subdigraph of $D$. Let $(i,j)$ be an arc in $D_1$. Then $r_{ij}^{D_1}$ need not be equal to $ r_{ij}^{D}$.\end{remark}
\begin{example} Consider $D$ and $D_1$ as in Figures 4.1 and 4.2. We have $r_{13}^D =  0.8125 + 0.4375 - 0.625 = 0.625 $ and
$r_{13}^{D_1} = 0.6389 + 0.3056 - 0.2778 = 0.6667.$
Hence $r_{13}^D \neq r_{13}^{D_1}$.
Note that $r_{13}^D$ (resp. $r_{13}^{D_1}$) is evaluated via the Moore-Penrose inverse of the Laplacian of $D$ (resp. $D_1$).
\end{example}
\begin{example}
Consider the strongly connected graph $D$ as in Figure 4.4. We have $r_{31}^D > d_{31}^D.$ This is a counterexample which shows that the condition 
``balanced" in Conjecture \ref{conj} cannot be removed.
\end{example}
\section{Proof of Theorem \ref{mainthm}}
\begin{lemma}\label{prop1}
Let $D_1, D_2 $ be connected balanced arc-disjoint digraphs and let $D_1 \cap D_2  $ be a single vertex. Suppose Conjecture \ref{conj} is true for both $D_1$ and $D_2$, then it is true for $D = D_1 \cup D_2$. In other words, 
%It is enough to prove the theorem  \ref{maintheorem} for the subcategory
%$\mathcal{C^{'}}= \{ D \mid D = D_1 \cup D_2 \}$ where 
%$D_1 \cup D_2$ is the union of $D_1$ and $D_2$ and $V(D_1) 
%\cap V(D_2) = \{ \mathrm{a} \ \mathrm{single}  \ \mathrm{vertex} \}$.
%By induction hypothesis, the theorem \ref{maintheorem} holds for $\mathcal{C}$.
$r_{ij}^{D} \leq 1$ for all arcs $(i,j) \in E(D).$
\end{lemma}
\begin{proof} Let the vertex sets of $D_1$ and $D_2$ be
 $V(D_1) = \{ 1, 2, ...., n-1, n \}$ and $V(D_2) = \{n,  n +1, ...., n + k \}$.
 Let $D = D_1 \cup D_2$. The vertex set $V$ of $D = D_1 \cup D_2$ is $\{1, 2,.., , n+k \}$ and
$V(D_1) \cap V(D_2) = \{ n \}$. Note that $|V_1|=n, |V_2| = k+1$.
Let $N = n + k$. Then $|V(D)| = N$. The digraphs $D_1$ and $D_2$ are both strongly connected and balanced. Hence $D$ is also strongly connected and balanced. Let $(i,j)$ be an arc.
We have two cases. 
Case (i) $D_1 \cap D_2 =  \{ i \} \ \mathrm{or} \  \{ j \}$, 
Case (ii) $D_1  \cap D_2 \neq \{ i \}$ and $\neq
\{ j \}.$\\
Suppose Case (i) holds. Let $(i,j)$ be an arc and $i = n$ or $j = n$. Then either $\{ i , j \} \subseteq V(D_1) \ \text{or} \ \{ i , j \} \subseteq  \ V(D_2)$.
 Without loss of generality, let $\{ i , j \} \subseteq V(D_1)$ and $i = 1,$ $j=n$. The digraph $D$ is balanced and hence $L = L(D)$ satisfies the condition that each row and each column adds up to zero and $\mathrm{rank}(L)=N-1$ and $L$ is a $\mathbb{Z}$-matrix (see reference \cite{BBG2}, section 2 (P7)). By Lemma \ref{mainlemma},
$$L^{\dagger}  =
%=\begin{bmatrix} A  & B \\ C  & D \end{bmatrix} 
\begin{bmatrix}  B^{-1} - \frac{ e e^{'} B^{-1} } {N}  - \frac{ B^{-1} e e^{'} }{ N }   &   - \frac{ B^{-1} e } {N} \\
  -\frac{ e^{'} B^{-1} } { N }  & 0  
\end{bmatrix} +   \frac{ e^{'} B^{-1} e 1 1^{'} }{ N^2 },   $$
where $B = L [ {\{ n \} }^c, {\{ n \} ^c} ]$. Let $J$ be the 
$(N-1) \times (N-1)$ all $1$-matrix. 
$\mathrm{Set} \ B^{-1} = C = (c_{ij}), \\ y =  e^{'} C = (y_i),  \ x = Ce = (x_i),$ $\frac{ e^{'} B^{-1} e 1 1^{'} }{ N^2 } = x_0 J,$ for some rational number $x_0.$
We then have 
$$  - \frac{ e e^{'} B^{-1} } {N} = -\frac{1}{N} \begin{bmatrix}
 y_1 & y_2 & \cdot & \cdot & \cdot & y_{N-1}\\
y_1 & y_2 & \cdot & \cdot & \cdot & y_{N-1}\\
\cdot & \cdot & \cdot  & \cdot & \cdot & \cdot \\
\cdot & \cdot & \cdot  & \cdot & \cdot & \cdot \\
y_1 & y_2 & \cdot & \cdot & \cdot & y_{N-1}
\end{bmatrix}, \  - \frac{ B^{-1} e e^{'} }{ N }  = -\frac{1}{N} \begin{bmatrix}  x_1 & x_1 & \cdot & \cdot & \cdot & x_1\\
x_2 & x_2 & \cdot & \cdot & \cdot & x_2\\
\cdot & \cdot & \cdot  & \cdot & \cdot & \cdot \\
\cdot & \cdot & \cdot  & \cdot & \cdot & \cdot \\
x_{N-1} & x_{N-1} & \cdot & \cdot & \cdot & x_{N-1}
\end{bmatrix},$$ where 
$x_i = \sum_{l=1}^{N-1} c_{il}, \ y_i = \sum_{l=1}^{N-1} c_{li},$ for all $i=1$
 to $N-1.$

Consider
$$r_{1n}^{D} = l_{11}^{\dagger} + l_{nn}^{\dagger} - 2 l_{1n}^{\dagger}
= c_{11} -  \frac{y_1}{N} + x_0 - \frac{x_1}{N}  + x_0 - ( - \frac{2}{N} x_1 + 2x_0 ) = c_{11} - \frac{y_1}{N}+ \frac{x_1}{N}.$$
Thus
\begin{equation} r_{1n}^{D} =   c_{11} + \frac{(x_1 - y_1)}{N}.\label{equation1}
\end{equation}\\
\smallskip
By Lemma \ref{lemma7}, the Moore Penrose inverse of  
$B=L [ {\{ n \} }^c, {\{ n \} ^c} ] =
  \begin{bmatrix} { L_1[ {\{ n \} }^c, {\{ n \} ^c} ] } & 0 \\
0 & { L_2 [ {\{ n \} }^c, {\{ n \} ^c} ] } \end{bmatrix}$ is given by\\
 $$B^{-1} ={ L [ {\{ n \} }^c, {\{ n \} ^c} ] }^{-1} = 
 \begin{bmatrix} { L_1[ {\{ n \} }^c, {\{ n \} ^c} ] }^{-1} & 0 \\
0 & { L_2 [ {\{ n \} }^c, {\{ n \} ^c} ] }^{-1} \end{bmatrix}.$$

\smallskip

By Lemma \ref{mainlemma},
we know that
 $$L_1^{\dagger}  =
=
\begin{bmatrix}  B_1^{-1} - \frac{ e e^{'} B_1^{-1} } {n}  - \frac{ B_1^{-1} e e^{'} }{ n }   &   - \frac{ B_1^{-1} e } {n} \\
  -\frac{ e^{'} B_1^{-1} } { n }  & 0 
\end{bmatrix} +   \frac{ e^{'} B_1^{-1} e 1 1^{'} }{ n^2 },  \mathrm{where} \ B_1 = L_1 [ {\{ n \} }^c, {\{ n \} ^c} ] . $$
 $$\mathrm{Set} \ B_1^{-1} = C^{*} = (c^{*}_{ij}), \ y^{*} =  e^{'} C^{*} = (y_i^{*}),  \ x^{*} = C^{*}e = (x_i^{*} ). \
\mathrm{We} \ \mathrm{ get} \ r_{1n}^{D_1} =   c^{*}_{11}+ \frac{(x_1^{*} - y_1^{*})}{n_1}.$$
From $B^{-1} =
 \begin{bmatrix} B_1^{-1} & 0 \\
0 & { L_2 [ {\{ n \} }^c, {\{ n \} ^c} ] }^{-1} \end{bmatrix},$ we get 
$c_{11} = c_{11}^{*}, \ y_1 = y_1^{*}, \ x_1 = x_1^{*}.$ Thus
\begin{equation}\label{equation2}
 r_{1n}^{D_1} =   c_{11}+ \frac{(x_1 - y_1)}{n}.
\end{equation}

By combining equations \ref{equation1} and \ref{equation2}, we get, as $N = n+ k$,
\begin{equation}\label{equation7}
r_{1n}^{D} =  c_{11} + \frac{n}{N} ( r_{1n}^{D_1} - c_{11} )  = \frac{  (n r_{1n}^{D_1} + k c_{11} )} {n+k} .
\end{equation}
Now, by Lemma \ref{lemma6}, 
\begin{equation}\label{equation8} c_{11} = \frac{ \mathrm{det} L[\{1,n\}^c, \{1,n \}^c] } {\kappa(D)}  \leq 1.
\end{equation}
Conjecture \ref{conj} is true for $D_1,$ so 
\begin{equation}\label{equation9} r_{1n}^{D_1} \leq 1.
\end{equation} 
We conclude from equations \ref{equation7}, \ref{equation8} and \ref{equation9} that
$
 r_{1n}^{D} \leq \frac{ n+ k}{n+k}  =1.$
 
Similarly,  when 
 $ \{ i , j \} \subseteq V(D_1),$  $(i,j) \in E(D_1)$ and $j = 1,$ $i=n,$ by using Lemma \ref{lemma6} we get
  $c_{11}   \leq 1.$ Hence $r_{n1}^{D} = \frac{  (n r_{n1}^{D_1} + k c_{11} )} {n+k}  \leq 1.$
So $r_{ij}^{D} \leq 1$ if either $D_1 \cap D_2 = \{ i \} \ \mathrm{or} \{ j\}.$\\
Suppose Case (ii) holds. i.e. $D_1  \cap D_2 \neq \{ i \}$ and $\neq
\{ j \}.$\\Without loss of generality, let $\{ i , j \} \subseteq V(D_1)$. We have $i, j \neq n$. Proceeding as before, we get
\begin{equation}\label{equation3}
r_{ij}^{D} = l_{ii}^{\dagger} + l_{jj}^{\dagger} - 2 l_{ij}^{\dagger}
= c_{ii} +  c_{jj} - 2  c_{ij} +  \frac{(x_i - y_i + y_j - x_j)}{N}
\end{equation} and
\begin{equation}\label{equation4}
r_{ij}^{D_1} = l_{ii}^{\dagger} + l_{jj}^{\dagger} - 2 l_{ij}^{\dagger}
= c_{ii} +  c_{jj} - 2  c_{ij} +  \frac{(x_i - y_i + y_j - x_j)}{n}.
\end{equation}

By combining equations \ref{equation1} and \ref{equation2}, we get, as $N = n+ k$,
\begin{equation}\label{equation5}
r_{ij}^{D} =  (c_{ii}+ c_{jj}  - 2  c_{ij}) + \frac{n}{N} ( r_{ij}^{D_1} - (c_{ii}+ c_{jj}  - 2  c_{ij}) )  = \frac{  (n r_{ij}^{D_1} + k  (c_{ii}+ c_{jj}  - 2  c_{ij} ) )} {n+k} 
\end{equation} and
$r_{ij}^{D} + r_{ji}^D = 2( c_{ii} +  c_{jj} -  c_{ij} -  c_{ji})$. By Lemmas  by \ref{BBG2 lemma} and \ref{lemma6}
we thus have $$r_{ij}^{D} + r_{ji}^D = 2 \frac{\mathrm{det}L[ \{i,j\}^c, \{i,j\}^c]} {\kappa(D)} \leq 2.$$ Hence $c_{ii} +  c_{jj} -   c_{ij} - c_{ji} \leq 1$. The determinant of a square matrix is the same as the determinant of its transpose, we get $$c_{ij} = \frac{(-1)^{i+j}} {\kappa(D)}\mathrm{det}L[ \{i,n\}^c, \{j,n\}^c] = \frac{(-1)^{i+j}} {\kappa(D)} \mathrm{det}( [L[ \{i, n\}^c, \{j, n \}^c] ]^{'}) = c_{ji}.$$
Hence \begin{equation}\label{equation10} c_{ii} +  c_{jj} -  2 c_{ij} \leq 1.
\end{equation} By equations \ref{equation9}, \ref{equation5}  and \ref{equation10},  we conclude that
$r_{ij}^D \leq \frac{n+k}{n+k} =1. $
\end{proof}

%\begin{prop}\label{prop2}
%Let $D_1, D_2 $ be connected balanced arc-disjoint digraphs and let $D_1 \cap D_2  $ be a single vertex. Suppose Conjecture \ref{conj} is true for both $D_1$ and $D_2$.
%It is enough to prove the theorem  \ref{maintheorem} for the subcategory
%$\mathcal{C^{'}}= \{ D \mid D = D_1 \cup D_2 \}$ where 
%$D_1 \cup D_2$ is the union of $D_1$ and $D_2$ and $V(D_1) 
%\cap V(D_2) = \{ \mathrm{a} \ \mathrm{single}  \ \mathrm{vertex} \}$.
%By induction hypothesis, the theorem \ref{maintheorem} holds for $\mathcal{C}$.
% Let $(i,j)$ be an arc.
%Then $r_{ij}^{D} \leq 1$.
%\end{prop}
%\section{Application to Directed Cactus Graphs}
We now present the proof of Theorem \ref{mainthm}.
\begin{proof}
Let $D = D_1  \cup D_2 \cup ... \cup D_{k-1} \cup D_{k}$ be a digraph in $\mathcal{C}$.  We prove Theorem \ref{mainthm} by using induction on $k$, the number of digraphs $D_i$.
Note that the union of any finite number of strongly connected and balanced digraphs is also strongly connected and balanced. Without loss of generality, we can assume that $G = D_1  \cup D_2 \cup ... \cup D_{k-1}$ is connected and also $G$ and $D_k$ have a single vertex in common. If $D$ is a union of two digraphs $D_1$ and $D_2$, then
$r_{ij}^{D}   \leq d_{ij}^D$ by Lemma \ref{prop1}.
The conjecture \ref{conj} is then true for $ G = D_1  \cup D_2 \cup ... \cup D_{k-1}$ by induction. Hence 
 Conjecture \ref{conj} is true for $D = G \cup D_k$ by Lemma \ref{prop1}. 
 \end{proof}
\begin{cor}[\cite{BBG2}, Theorem 3.5]
Let $D $ be a directed cactus. Then
 $r_{ij}^D \leq d_{ij}^{D}$ for all $i$ and $j.$
\end{cor}
\begin{proof}
Let $D$ be a directed cactus. Then 
$D = D_1 \cup D_2 \cup ..... \cup D_k,  k \in \mathbb{N},$ where each $D_i$ is a balanced directed cycle and for all $i,$ $1 < i \leq k,$ $D_{i} \cap (D_1  \cup D_2 \cup ... \cup D_{i-1} )$ being a single vertex. We know that Conjecture\ref{conj} is true for $D_i$ by Lemma \ref{BBG lemma}. Hence $r_{ij}^D \leq d_{ij}^{D}$ by Theorem \ref{mainthm}.
\end{proof}
\section{Examples}
\begin{center}

\begin{center}
Figure 4.1. Strongly Connected and Balanced Digraph $D$ on $8$ vertices 
\end{center}
\vskip 15pt 
\begin{center}
\begin{tikzpicture} [scale=.8,auto=left,every node/.style={circle},decoration={markings,mark=at position 0.5 with {\arrow[scale=0.7]{triangle 60}}}
    ]
  
  \node[circle,draw] (n1) at (1,3) {v1} ;
  \node[circle,draw] (n3) at (3,5)  {v3} ;
  \node[circle,draw] (n2) at (3,1) {v2} ;
   \node[circle,draw] (n4) at (5,5) {v4}  ;
  \node[circle,draw] (n6) at (7,3)  {v6} ;
  \node[circle,draw] (n5) at (5,1)  {v5} ;
  \node[circle,draw] (n7) at (9,5)  {v7} ;
    \node[circle,draw] (n8) at (9,1)  {v8} ;

  \foreach \from/\to in {n1/n3,n2/n1, n3/n4,n5/n2,n4/n6,n6/n5,n6/n7,n7/n8, n8/n6}
    \draw[postaction={decorate}] (\from) -- (\to);     
    
   \draw
    (n2) edge[ bend left=15,postaction={decorate}] (n3)
    (n3) edge[ bend left=15,postaction={decorate}] (n2);  
\end{tikzpicture}
\end{center}
\begin{center} 
Figure 4.2 $D_1$ \ \ \ \ \ \ \ \ \ \ \  \ \ \ \ \ \ \ \ \ \ \ \ \ \ \ \ Figure 4.4 $D_2$
\end{center}
\begin{center}
\begin{tikzpicture}
   [scale=.8,auto=left,every node/.style={circle},decoration={markings,mark=at position 0.5 with {\arrow[scale=0.7]{triangle 60}}}
    ]
  
  \node[circle,draw] (n1) at (1,3) {v1} ;
  \node[circle,draw] (n3) at (3,5)  {v3} ;
  \node[circle,draw] (n2) at (3,1) {v2} ;
   \node[circle,draw] (n4) at (5,5) {v4}  ;
  \node[circle,draw] (n6) at (7,3)  {v6} ;
  \node[circle,draw] (n5) at (5,1)  {v5} ;
  \foreach \from/\to in {n1/n3,n2/n1, n3/n4,n5/n2,n4/n6,n6/n5}
    \draw[postaction={decorate}] (\from) -- (\to);     
    
   \draw
    (n2) edge[ bend left=15,postaction={decorate}] (n3)
    (n3) edge[ bend left=15,postaction={decorate}] (n2); 
\end{tikzpicture}
\qquad
\begin{tikzpicture}
 [scale=.8,auto=left,every node/.style={circle},decoration={markings,mark=at position 0.5 with {\arrow[scale=0.7]{triangle 60}}}
    ]

  \node[circle,draw] (n6) at (1,3)  {v6} ;
  \node[circle,draw] (n7) at (3,5)  {v7} ;
    \node[circle,draw] (n8) at (3,1)  {v8} ;

  \foreach \from/\to in {n6/n7,n7/n8, n8/n6}
    \draw[postaction={decorate}] (\from) -- (\to);    
\end{tikzpicture}
\end{center}
\begin{equation*}
L =   \left[ \begin{array}{cccccccc}
1& 0 & -1& 0 & 0 & 0 & 0 & 0 \\
-1 & 2 & -1 & 0 & 0 & 0 & 0 & 0 \\
0 & -1 &  2 & -1 & 0 & 0 & 0 & 0 \\
0 & 0 & 0 & 1 & 0 & -1 & 0 & 0 \\
0 & -1&  0 & 0 & 1 & 0 & 0 & 0 \\
0 & 0 & 0 & 0 & -1 & 2 & -1& 0 \\
0 & 0 & 0 & 0 & 0 & 0 & 1 & -1 \\
0 & 0 & 0 & 0 & 0 & -1 & 0 &  1
\end{array}
\right],\ 
\end{equation*}
\begin{equation*}
L_1 = \left[ \begin{array}{cccccc}
1 & 0 & -1 & 0  & 0 & 0  \\
-1 & 2 & -1 & 0 & 0 & 0 \\ 
0 & -1 & 2 & -1 & 0 & 0 \\
0 & 0 & 0 & 1 & 0 & -1 \\ 
0 & -1 & 0 & 0 & 1 & 0 \\
0 & 0 & 0 & 0 & -1 & 1 
\end{array}
\right], \
L_2 = \left[ \begin{array}{ccc}
1 & -1 & 0 \\
0 & 1 & -1 \\
-1 & 0 & 1\\
\end{array}
\right]
\end{equation*}

\vskip 10pt

\begin{equation*}
L^{\dagger} =   \left[ \begin{array}{cccccccc}
 0.8125 & -0.0625 & 0.3125 & 0.1875 & -0.3125 & -0.1875 & -0.3125 & -0.4375 \\
 0.3125 & 0.4375 & 0.3125 & 0.1875 & -0.3125 &  -0.1875 & -0.3125 & -0.4375 \\
 -0.0625 & 0.0625 & 0.4375 & 0.3125 & -0.1875 & -0.0625 & -0.1875 & -0.3125 \\
 -0.3125 & -0.1875 & -0.3125 & 0.5625 & 0.0625 & 0.1875 & 0.0625 & -0.0625 \\
 0.1875 & 0.3125 & 0.1875 & 0.0625 & 0.5625 & -0.3125 & -0.4375 & -0.5625 \\
 -0.1875 & -0.0625 & -0.1875 & -0.3125 & 0.1875 & 0.3125 & 0.1875 & 0.0625 \\
 -0.4375 & -0.3125 & -0.4375 & -0.5625 & -0.0625 & 0.0625 & 0.9375 & 0.8125 \\
 -0.3125 & -0.1875 & -0.3125 & -0.4375 & 0.0625 & 0.1875 & 0.0625 & 0.9375 
\end{array}
\right] \ \
\end{equation*}

\begin{equation*}
L_1^{\dagger} = \left[ \begin{array}{cccccc}
0.6389 & -0.1944 & 0.1389 & -0.0278 & -0.3611 & -0.1944 \\
 0.1389 & 0.3056 & 0.1389  & -0.0278 & -0.3611 & -0.1944 \\
 -0.1944 & -0.0278 & 0.3056 & 0.1389 &  -0.1944 & -0.0278 \\
 -0.3611 & -0.1944 & -0.3611 & 0.4722 & 0.1389 & 0.3056 \\
 -0.0278  & 0.1389 & -0.0278 & -0.1944 & 0.4722 & -0.3611 \\
 -0.1944 &  -0.0278 & -0.1944 & -0.3611&  0.3056 & 0.4722 
 \end{array}
 \right], 
 \end{equation*}
 \begin{equation*}
 L_2^{\dagger} = \left[ \begin{array}{ccc}
 0.3333& 0.3333 & 	0.3333	\\
-0.1666 &	0.33333333 &	-0.1666\\	
-0.1666 &	-0.6666 &	-0.1666
 \end{array}
 \right].
 \end{equation*}
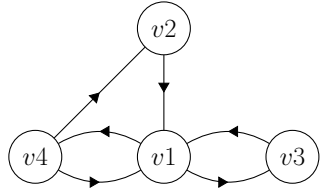
Figure 4.4. Strongly Connected Digraph $D$ on $4$ vertices
\end{center}
{\scalefont{3.0}

\begin{center}
\begin{tikzpicture} [decoration={markings,mark=at position 0.5 with {\arrow[scale=0.7]{triangle 60}}}
    ]
 \tikzset{vertex/.style = {scale=.4, shape=circle,draw}}
        \tikzset{edge/.style = {->,> = latex'}}
\node[vertex](v1) {$v1$};
\node[vertex,right=of v1] (v3) {$v3$};  
\node[vertex,left=of v1] (v4) {$v4$};  
\node[vertex,above=of v1] (v2) {$v2$};  
\draw
(v1) edge[ bend right=30,auto=right,postaction={decorate}] node {} (v3)
(v3) edge[ bend right=30,auto=right,postaction={decorate}] node {} (v1)
(v4) edge[bend right=30,auto=right,postaction={decorate}] node {} (v1)
(v1) edge[ bend right=30,auto=right,postaction={decorate}] node {} (v4)
(v2) edge[auto=below,postaction={decorate}] node {} (v1)
(v4) edge[postaction={decorate}] node {} (v2);   
\end{tikzpicture}
\end{center}
}
The indegree at the vertex $v_1$ is $3$ and the outdegree at the vertex $v_1$ is 2.
Hence the above graph $D$ is not balanced. 
The degree matrix $D$ and the adjacency matrix $A$ are given by\\
\begin{equation*}
D =  \left[ \begin{array}{cccc}
2 & 0 & 0 & 0 \\
0 & 1 & 0 & 0 \\
0 & 0 & 1 & 0 \\
0 & 0 & 0 & 0 
\end{array} \ \ 
 \right],
 \ \ 
A=  \left[ \begin{array}{cccc}

0 & 0 & 1 & 1 \\
1 & 0 & 0 & 0 \\
1 & 0 & 0 & 0 \\
1 & 1 & 0 & 0 
\end{array}
\right]
\end{equation*}

\smallskip
The Laplacian matrix $L(D)$ and the Moore-Penrose inverse of $L(D)$ are given by
\smallskip

\begin{equation*}
L(D) = D - A =   \left[ \begin{array}{cccc}
2& 0 & -1 & -1\\
-1 & 1 & 0 & 0 \\
-1 & 0 & 1 & 0 \\
-1 & -1 & 0 & 2 
\end{array}
\right], \ \
L(D)^{\dagger} = \left[ \begin{array}{cccc}
0.2 & -0.275 & -0.05 & -0.025\\
0 & 0.625 & -0.25 & -0.125 \\
-0.2 & -0.475 & 0.5500 & -0.225 \\
0 & 0.125 & -0.25 & 0.375
\end{array}
\right]
\end{equation*}
\smallskip
 $$ \mathrm{We} \ \mathrm{have} \ r_{31} = l_{33}^{\dagger} + l_{11}^{\dagger} - 2 l_{31}^{\dagger} = 0.2 + 0.55+0.4
= 1.15.$$
Note $(3,1)$ is an arc. Hence $d_{31} = 1.$
We have $r_{31} > d_{31}$.
\\
Hence the graph $D$ is a counterexample which shows that the condition 
``balanced" in Conjecture \ref{conj} cannot be removed. 

%\begin{remark}

 \subsection{Concluding Remark}
The natural questions to ask are\\
(i) In Figure 4.2  the digraph $D_1 = D \cap D^{'},$ where the digraph $D$ is a balanced triangle and the digraph $D^{'}$ is a balanced pentagon, numerically we have verified Conjecture \ref{conj2} for $D_1.$ However an abstract proof is not yet known. More precisely, If $D$ and $D^{'}$ both are connected, balanced digraphs such that $D \cap D^{'}$ have only two vertices in common and Conjecture \ref{conj2} is true for both $D$ and $D^{'}$  then whether it is true for the union  $D \cup D^{'}$? 
\\
(ii) If $D_1, D_2,...,D_k$ are connected, balanced digraphs with $D_1 \cup D_2 \cup ....\cup D_k$ is connected, $D_{i} \cap D_j$ being at most a single vertex, for all $i \neq j,$ $1 \leq i,j \leq k$ and Conjecture  \ref{conj2} is true for all $D_i$, $i=1$ to $k$, then whether it is true for  $D=D_1 \cup D_2 \cup ....\cup D_k?$ i.e., whether $D$ satisfies $r_{ij}^D \leq d_{ij}^D$ for all $i$ and $j$?
The methods in the proof of our main theorem [Theorem \ref{mainthm}] is not applicable to answer (i) and (ii). We need alternative methods to answer these questions.

\section*{Acknowledgement}
We would like to thank Bharathidasan University, IISER, Thiruvananthapuram and  San Jose State University for providing the excellent working conditions.
%We would also like to thank Wasin So for helpful discussions, comments and suggestions.
We thank Ananth Narayanan for some helpful comments.\\

%Data Availability : Authors can confirm that all relevant data are included in the article.\\
%Funding, Code Availability and Materials Availability : Not applicable.\\
%Conflict of Interest : The authors declare that there is no conflict of interest.

\end{document}